\documentclass[ejs]{imsart}

\RequirePackage{amsthm,amsmath, amssymb}
\RequirePackage[numbers]{natbib}
\RequirePackage[colorlinks,citecolor=blue,urlcolor=blue]{hyperref}
\RequirePackage{cleveref}
\RequirePackage{dsfont}
\RequirePackage[usenames,dvipsnames]{xcolor}
\RequirePackage{enumerate}

\doi{10.1214/154957804100000000}
\pubyear{0000}
\volume{0}
\firstpage{1}
\lastpage{1}

\startlocaldefs

\newtheorem{lemma}{Lemma}
\newtheorem{prop}[lemma]{Proposition}
\newtheorem{rem}[lemma]{Remark}
\newtheorem{thm}[lemma]{Theorem}
\newtheorem{definition}[lemma]{Definition}
\newtheorem{conj}[lemma]{Conjecture}

\DeclareMathOperator{\Var}{Var}
\DeclareMathOperator{\Cov}{Cov}

\DeclareMathOperator{\Lip}{Lip}

\newcommand{\NN}{\mathbb{N}}
\newcommand{\RR}{\mathbb{R}}
\newcommand{\ZZ}{\mathbb{Z}}
\newcommand{\EE}{\mathbb{E}}
\newcommand{\PP}{\mathbb{P}}

\newcommand{\one}{\mathds{1}}
\newcommand{\singleest}[1]{\widehat{p}_{\ifblank{#1}{b_n}{#1}}(u_n)}
\newcommand{\multest}[1]{\widehat{\mathbb{P}}_{T, \ifblank{#1}{b_n}{#1}}(u_n)}
\newcommand{\pconv}{\stackrel{\text{P}}{\longrightarrow}}
\newcommand{\dconv}{\stackrel{\text{d}}{\longrightarrow}}
\newcommand{\Xvec}[1]{\mathbf{X}_{\ifblank{#1}{t}{#1}+T}}
\newcommand{\blarge}{b_{n, 1}}
\newcommand{\bsmall}{b_{n, 2}}
\newcommand{\belowprob}[1]{\mathbb{P}_{X, u}(#1)}

\endlocaldefs

\begin{document}

\begin{frontmatter}

\title{Long Memory of Max-Stable Time Series as Phase Transition: Asymptotic Behaviour of Tail Dependence Estimators}
\runtitle{Long Memory of Max-Stable Time Series as Phase Transition}


\author{\fnms{Marco} \snm{Oesting}\ead[label=e1]{marco.oesting@mathematik.uni-stuttgart.de}}
\address{Stuttgart Center of Simulation Science \& Institute for Stochastics and Applications, \\
University of Stuttgart, \\ Pfaffenwaldring 5a, 70569 Stuttgart, Germany\\ 
	\printead{e1}}
\and
\author{\fnms{Albert} \snm{Rapp}\thanksref{c1} \ead[label=e2]{albert.rapp@uni-ulm.de}}
\address{Universität Ulm, Institut für Stochastik, \\ Helmholtzstraße 18, 89069 Ulm, Germany\\ 
	\printead{e2}}

\thankstext{c1}{Corresponding author}

\runauthor{Oesting \& Rapp}

\begin{abstract}
	In this paper, we consider a simple estimator for tail dependence coefficients of a max-stable time series and show its asymptotic normality under a mild condition.
	The novelty of our result is that this condition does not involve mixing properties that are common in the literature.
	More importantly, our condition is linked to the transition between long and short range dependence (LRD/SRD) for max-stable time series.
    This is based on a recently proposed notion of LRD in the sense of indicators of excursion sets which is meaningfully defined for infinite-variance time series.
    In particular, we show that asymptotic normality with standard rate of convergence and a function of the sum of tail coefficients as asymptotic variance holds if and
    only if the max-stable time series is SRD.
\end{abstract}





\end{frontmatter}

\section{Introduction}
\allowdisplaybreaks

In extreme value analysis, one is typically interested in events which are extreme in the sense that a high threshold $u$ is exceeded, i.e.~events of the form $X_0 > u$ where $X_0$ is part of a time series $(X_t)_{t \in \mathbb{Z}}$.
Although $\lim_{u \to \infty} \PP(X_0 > u) = 0$, one studies the behaviour as $u \to \infty$ because the limiting probability
$$ \chi_h := \lim_{u \to \infty} \PP(X_h >u \mid X_0 > u) \in [0,1] $$
might be positive, provided that it exists. 
Assuming that the limit indeed exists, $\chi_h$ denotes the limiting probability that, conditional on an extreme event at time $t=0$, another extreme event occurs within $h \in \ZZ$ time steps.
If $\chi_h > 0$, the random variables $X_0$ and $X_h$ are called asymptotically dependent, while they are called asymptotically independent if $\chi_h=0$.
For a longer discussion on modelling asymptotic dependence using tail dependence coefficients such as $\chi_h$ let us refer to Chapter 9.5 in \cite{Beirlant.04}.

If $(X_t)_{t \in \ZZ}$ is a max-stable time series, it is well-known that $\chi_h$ always exists and is of the form $ \chi_h = 2 - \theta_h$, $h \in \ZZ, $ where $\theta_h \in [1,2]$ denotes the corresponding pairwise extremal coefficient. The latter coefficient was first defined in \cite{smith90} and is an important characteristic quantity in extreme value statistics. 
In the max-stable case, it has been shown that $\theta_h=2$, or, equivalently, $\chi_h=0$ if and only if $X_0$ and $X_h$ are fully independent. 
Thus, components of a max-stable vector are either asymptotically dependent or fully independent and the pairwise extremal coefficients can be interpreted as a measure for the strength of dependence. 
Hence, the pairwise extremal coefficient has gained a lot of attention and has been thoroughly studied in the literature, e.g.~in \cite{coles-and-tawn96}, \cite{coles93} and \cite{schlather-tawn-02}.

In this paper, we will consider the case that $(X_t)_{t \in \ZZ}$ is a stationary max-stable time series with $\alpha$-Fr\'echet, i.e.~heavy-tailed, marginal distributions and focus on the estimation of limits of conditional exceedance probabilities of the form $\mathbb{P}(X_{t_1} > u, \ldots, X_{t_N} > u \mid X_0 > u)$ as $u \to \infty$ where $0 < t_1 < \ldots < t_N$ for some $N \in \NN$.
This is well-defined as the existence of the limit
\begin{align}
	\label{eq: chi_T def}
	\chi_T 
	:&= 
	\lim_{u \to \infty} \mathbb{P}(X_{t} > u \text{ for all } t \in T \mid  X_0 > u) \\
	&= 
	\lim_{u \to \infty} \frac{\mathbb{P}(X_{t} > u \text{ for all } t \in T )}{\mathbb{P}(X_0 > u)} \in [0, 1],
\end{align}
is ensured for any set $T = \{0, t_1,\ldots,t_N\}$ by \Cref{lem:aux-limits} in the appendix.

In practice, estimation has to be based on a single realization of the time series up to time $n$, i.e. on dependent observations $X_1, \ldots, X_n$. 
A straight-forward estimator can be constructed by replacing the desired probabilities in \eqref{eq: chi_T def} with relative frequencies, i.e.~a natural estimator of $\chi_T$ is given by
\begin{align}
	\widehat 
	\chi_{T ,n} 
	&:= 
	\frac{
		\multest{n} 
	}{
		\singleest{n} 
	} :=
	\frac{
		\frac 1 n \sum_{t=1}^{n} \one\{u_n^{-1} \Xvec{}  \in C_{T}\}
	}{
		\frac 1 n \sum_{t=1}^{n} \one\{u_n^{-1} \Xvec{} \in C_{0}\}
	}
	={} 
	\frac{
		\sum_{t=1}^{n} \one\{u_n^{-1} \Xvec{}  \in C_{T}\}
	}{
		\sum_{t=1}^{n} \one\{u_n^{-1} \Xvec{} \in C_{0}\}
	},
	\label{eq: estimator}
\end{align}
where $\Xvec{} = \big( X_{t+s}\big)_{s \in T}$, for each $t \in \mathbb{Z}$, $(u_n)_{n \in \mathbb{N}}$ is a sequence such that $u_n \rightarrow \infty$ as $n \rightarrow \infty$ and
\begin{align}
	C_0 &= 
	\Big\{ (x_0, \ldots, x_{t_N}) > 0 \, \Big\vert \, x_0 > 1 \Big\}, \\
	C_{T} 
	&= 
	\Big\{ (x_0, \ldots, x_{t_N}) > 0 \, \Big\vert \, x_i > 1 \text{ for all } i \in T\Big\} 
	\subset 
	\mathbb{R}^{t_N + 1}.
	\label{eq: C-sets}
\end{align}
Note that we need $u = u_n$ to depend on $n$ with $u_n \to \infty$ in order to get a good approximation of the limiting quantity $\chi_T$. 
Furthermore, our estimator uses observations up until time point $n + t_N$ rather than just $n$. 
Asymptotically, the use of $t_N$ more observations has no influence but eases notation.
Unfortunately, the larger $u_n$ gets, the less data can be effectively used for our estimation since we do not count those data that are less than $u_n$. 
Thus, the effective size of the sample that we use is only $\sum_{t=1}^{n} \one\{X_t > u_n\}$ that is, on average, $n \PP(X_0 > u_n)$. 
To make sure that the effective sample size grows to $\infty$ as $n \to \infty$, a natural assumption is $\lim_{n \to \infty} n \PP(X_0 > u_n) = \infty$.

The main result of our paper is that we prove asymptotic normality of our estimator \eqref{eq: estimator} in \Cref{thm: asym-normal-chi-SRD}.
The novelty of this result is that \Cref{thm: asym-normal-chi-SRD} characterizes the dependence of the underlying time series in terms of the weak assumption $\sum_{t=1}^\infty \chi_t = \sum_{t=1}^\infty (2-\theta_t) < \infty$ which is equivalent to the finiteness of the asymptotic variance of the estimator.
This assumption is weaker than common assumptions in terms of strong mixing properties involving $\alpha$-mixing coefficients.

More importantly, our condition $\sum_{t=1}^\infty (2-\theta_t) < \infty$ opens a connection to the phenomenon of long range dependence (LRD).
While this phenomenon has seen increasing relevance in other fields, there are only few results relating to LRD of max-stable time series.
Clearly, this is the case because LRD was previously investigated mainly for finite-variance processes in terms of auto-covariance functions.
However, there have been recent investigations into LRD for infinite-variance processes.
For example, \cite{kulik2021long} defines LRD in terms of the covariance of indicators of excursion sets, cf.~\Cref{def: LRD-KS}.
The advantage of this definition is that it allows to investigate LRD for infinite-variance processes because the covariance of indicators is always well-defined as the indicator functions are bounded.
Furthermore, \cite{makogin2021long} have recently shown that max-stable time series are LRD in the sense of \Cref{def: LRD-KS} iff $\sum_{t=1}^\infty (2-\theta_t) = \infty$.
Hence, our condition $\sum_{t=1}^\infty (2-\theta_t) < \infty$ relates to the SRD case.

In addition, our condition marks the point of a phase transition.
Recently, Chapter 9 of \cite{samorodnitsky2016stochastic} describes the phenomenon of LRD as a phase transition after which the convergence rate in a limit theorem changes magnitude.
In this paper, we show that the limiting variance of our estimator's denominator undergoes such a phase transition as $X$ becomes LRD.
Hence, our paper takes the first steps to investigate LRD in an extreme value statistics setting.
Furthermore, we round off our paper by discussing ways to extend our results into the realm of LRD even further.

Our paper is structured as follows.
\Cref{sec: max-stable} reviews well-known facts about max-stable time series that we need in order to prove our main result.
\Cref{sec: var-behavior-SRD} investigates the limiting behavior of $\Var\big(\singleest{}\big)$ and $\Var\big(\, \multest{}\big)$ for arbitrary integer sequences $b_n \rightarrow \infty$ under the SRD assumption $\sum_{t=1}^\infty (2-\theta_t) < \infty$.
In \Cref{sec: asym-normality-SRD},  we prove \Cref{thm: asym-normality-SRD} that shows the asymptotic normality of $\singleest{}$ and $\multest{}$ in the SRD-case. 
This result is then used to prove our main result.
\Cref{sec: conjectures-LRD} proves the aforementioned phase transition and discusses conjectures on extending \Cref{thm: asym-normal-chi-SRD} to the LRD case.
Finally, auxillary lemmas and the proof of \Cref{thm: asym-normality-SRD} can be found in \Cref{sec: proof-asym-normality-SRD}.

\section{Max-Stable Time Series}
\label{sec: max-stable}

Let us review properties of max-stable time series and their extremal coefficients. 
These properties are crucial to prove asymptotic normality of the estimators $\singleest{}$ and $\multest{}$ which is a key step to prove our main result, \Cref{thm: asym-normal-chi-SRD}.

Let $X = (X_t)_{t \in \ZZ}$ be a stationary max-stable time series with $\alpha$-Fr\'echet margins. 
Then, by the spectral representation according to \cite{de1984spectral},
one can write
$$ X_t = \max_{i \in \NN} \Gamma_i^{-\alpha} W_t^{(i)}, \quad t \in \ZZ,$$
where $\Gamma_1 < \Gamma_2 < \ldots$ are the arrival times of a unit rate Poisson process on $(0,\infty)$ and, independently of this Poisson process, $W^{(1)}, W^{(2)}, \ldots$ are independent copies of some stochastic process $W = (W_t)_{t \in \ZZ}$ satisfying $\EE(W_t^{\alpha})=1$ for all $t \in \ZZ$. 
A well-known consequence of this is that for any finite index set $T \subset \mathbb{Z}$, one obtains
\begin{equation} \label{eq:fdd}
\PP( X_t \leq x_t \text{ for all } t \in T) 
= \exp\left( - \EE\left\{ \max_{t \in T} \left(\frac{W_t}{x_t}\right)^\alpha\right\}\right), \quad (x_t)_{t \in T} \in (0,\infty)^{|T|}.
\end{equation}
It can be seen that the joint c.d.f.\ of $(X_t)_{t \in T}$ can be used to define a measure $\mu_T$ on $E_T = [0,\infty)^{|T|} \setminus \{0\}$, the so-called exponent measure, via
$$ \mu_T([0,x]^c) = - \log \PP( X_t \leq x_t \text{ for all } t \in T) = \EE\left\{ \max_{t \in T} \left(\frac{W_t}{x_t}\right)^\alpha\right\}$$
for $x = (x_t)_{t \in T} \in (0,\infty)^{|T|}$. 
By construction, the exponent measure $\mu_T$ satisfies $\mu_T(A) < \infty$ for any measurable set $A \subset E_T$ bounded away from $0 \in \RR^{|T|}$ and is homogeneous of order $-\alpha$, i.e.\ $\mu_T(cA) = c^{-\alpha} \mu_T(A)$ for all $c>0$ and the same sets $A$ as before. Using this definition, Equation \eqref{eq:fdd} implies that
$$ \lim_{u \to \infty} u^\alpha \PP( (X_t)_{t \in T} \in [0,ux]^c ) = \mu_T([0,x]^c), \quad x = (x_t)_{t \in T} \in (0,\infty)^{|T|},$$
or, more generally,
$$ \lim_{u \to \infty} u^\alpha \PP( (X_t)_{t \in T} \in uA ) = \mu(A) $$
for any measurable set $A \subset E_T$ bounded away from $0 \in \RR^{|T|}$ and satisfying $\mu(\partial A) =0$. In particular, this also ensures the existence of the limit
$$ \chi_T = \frac{\mu_{\{0\} \cup T} ((1,\infty) \times \ldots \times (1,\infty)) }{\mu_{\{0\}}((1,\infty))} = \mu_{\{0\} \cup T} ((1,\infty) \times \ldots \times (1,\infty)),$$
where we used that $\mu_{\{0\}}([1,\infty)) = \lim_{u \to \infty} u^\alpha \PP(X_0 > u) = 1$ according to the first part of Lemma \ref{lem:aux-limits}.

A dependence measure that is closely related to $\chi_T$ is the extremal coefficient
$$ \theta_T = \mu_T([0,1]^c) = \EE\left\{ \max\nolimits_{t \in T} W_t^\alpha\right\} \in \big[1,\ |T|\, \big]$$
such that
$$ \PP(X_t \leq x \text{ for all } t \in T) = \exp\left(- \theta_T x^{-\alpha}\right), \quad x >0.$$
In the case that $T=\{0, h\}$ with $h \neq 0$, we have that
 \begin{align*}
  \chi_T ={}& \mu_T((1,\infty) \times (1,\infty)) \\
  ={}& \mu_T((1,\infty) \times [0,\infty)]) + 
 \mu_T([0,\infty) \times (1,\infty)) - \mu_T([0,1]^c)\\
 ={}& \mu_{\{0\}}((1,\infty)) + \mu_{\{h\}}((1,\infty)) - \mu_T([0,1]^c) = 2 - \theta_T =: 2 - \theta_h.
 \end{align*}
Here, we denoted $\theta_T = \theta_h$ for $T = \{0, h\}$.
For more general finite index sets $T$, we obtain
\begin{align*}
	\chi_T ={}& \lim_{u \to \infty} \bigg[\frac{\PP(X_t > u\ \forall\ t \in T \setminus \{0\}, X_0 > u)}{\PP(X_0 > u)} - \underbrace{\PP(X_{t} > u \ \forall\ t  \in T \setminus \{0\})}_{\to 0}\bigg] \\
	={}& \lim_{u \to \infty} \frac{\Cov(\one\{X_t > u \, \forall t \in T \setminus \{0\}\}, \one\{X_0 > u\})}{\PP(X_0 > u)} \\
	={}& \sum_{\substack{C \subset \{0\} \cup (T \setminus \{0\}) \\ 0 \in C, C \cap T \neq \emptyset}}   (-1)^{|C|} \left( 1 + \theta_{C \cap  (T \setminus \{0\})} - \theta_C \right). 
\end{align*}
Let us note that the last equality follows from \Cref{lem:aux-limits} with $A = \{ 0 \}$ and $B = T \setminus \{ 0 \}$.

Finally, let us review the notion of LRD that was discussed in the introduction.
This notion is defined in terms of the covariance of indicators of excursion sets.

\begin{definition}[\cite{kulik2021long}]
	\label{def: LRD-KS}
	A real-valued stationary stochastic process $X = \{ X(t), \, t \in T \}$ where $T$ is an unbounded subset of $\mathbb{R}$ is short range dependent (SRD) if
	\begin{align}
		\int_{T} &\int_{\mathbb{R}} \int_{\mathbb{R}} \Big\vert \Cov (\mathds{1} \{ X(0) > u \}, \mathds{1} \{ X(t) > v \})\Big\vert \, \mu(\mathrm{d} u) \, \mu(\mathrm{d} v)\, \mathrm{d} t < \infty \label{eq: LRD_Spodarev}
	\end{align}
	for any finite measure $\mu$ on $\mathbb{R}$. Otherwise, i.e. if there exists a finite measure $\mu$ such that the integral in inequality \eqref{eq: LRD_Spodarev} is infinite, $X$ is long range dependent. 
	For stochastic processes in discrete time, the integral $\int_T\, \mathrm{d} t$ should be replaced by the summation $\sum_{t \in T}$.
\end{definition}

Notice that this definition is defined even for infinite-variance processes.
Hence, it is also applicable to the max-stable time series we are investigating in this paper.
In fact, due to \cite{makogin2021long} it is known that a max-stable time series is LRD in the sense of \Cref{def: LRD-KS} iff $\sum_{t=1}^\infty (2-\theta_t) = \infty$.
Furthermore, it is worth pointing out that this definition is invariant under strictly monotonic transformations.
This is desirable in an extreme value setting because it means that this notion of LRD does not depend on whether we use $\alpha$-Fréchet, Gumbel or Weibull margins. 

\section{Variance behavior}
\label{sec: var-behavior-SRD}
Let us consider the asymptotic variance behavior of $\singleest{n}$ and $\multest{n}$.
For our purposes it is suitable to replace $n$ by a more flexible sequence $b_n$ that fulfills $b_n \rightarrow \infty$ as $n \rightarrow \infty$. 
We will denote the corresponding estimators by $\singleest{}$ and $\multest{}$, respectively. 
For the estimator $\singleest{}$, we will see that the typical rate of convergence $\sqrt{b_n p(u_n)}$ holds if and only if the time series is short range dependent in the sense of \Cref{def: LRD-KS}, i.e. $\sum_{t=1}^\infty (2-\theta_t) < \infty$.

\begin{prop} \label{prop:pun}
	Let $(X_t)_{t \in \mathbb{Z}}$	be a max-stable stationary time series
	with $\alpha$-Fr\'echet margins and extremal coefficients $(\theta_t)_{t \in \mathbb{Z}}$.
	Then, for any real sequence $u_n \to \infty$ and any integer sequence $b_n \to \infty$, we have
	\begin{align}
		\lim_{n \to \infty} b_n p(u_n) \mathbb{E}\left[ \left( \frac{\singleest{}}{p(u_n)} - 1 \right)^2 \right]
		&= 
		\lim_{n \to \infty} \Var\left(\sqrt{b_n p(u_n)} \frac{\singleest{}}{p(u_n)} \right) \notag \\
		&= 
		1 + 2 \sum_{t=1}^\infty (2-\theta_t).
		\label{eq: var_p_hat}
	\end{align}
	In particular, the variance is finite if and only if $\sum_{t=1}^\infty (2-\theta_t) < \infty$.	
\end{prop}
\begin{proof}
	Using the stationarity, we calculate that
	\begin{align*}
		& \Var\left(\sqrt{b_n p(u_n)} \frac{\singleest{}}{p(u_n)} \right) \\
		=& \frac {1}{b_n p(u_n)} \sum_{j=1}^{b_n} \sum_{k=1}^{b_n} \Cov(\one\{X_j > u_n\}, \one\{X_k > u_n\}) \\
		=& \frac {\Var(\one\{X_0 > u_n\})} {p(u_n)} 
		+ \frac 2 {b_n p(u_n)} \sum_{t=1}^{b_n-1} (b_n - t) \Cov(\one\{X_0 > u_n\}, \one\{X_t > u_n\}) \\
		=& \frac {\Var(\one\{X_0 > u_n\})}{p(u_n)}
		+ \frac 2 {p(u_n)} \sum_{t=1}^{b_n-1} \left(1 - \frac{t}{b_n}\right) \Cov(\one\{X_0 > u_n\}, \one\{X_t > u_n\})	
	\end{align*}	
	\textbf{First Case:} $\sum_{t=1}^\infty (2 - \theta_t) < \infty$\\
	In this case, we can analyze the two terms in
	\begin{align*}
		& \Var\left(\sqrt{b_n p(u_n)} \frac{\singleest{}}{p(u_n)} \right) \\
		={}& \frac { \Var(\one\{X_0 > u_n\})}{p(u_n)}
		+ \frac 2 {p(u_n)} \sum_{t=1}^{b_n-1} \left(1 - \frac{t}{b_n}\right) \Cov(\one\{X_0 > u_n\}, \one\{X_t > u_n\}). 	
	\end{align*}
	separately.\\[1mm]
	Using that 	$\Var(\one\{X_0 > u_n\}) = \mathbb{P}(X_0 > u_n) - \mathbb{P}(X_0 > u_n)^2 = p(u_n) - p(u_n)^2$, the first term is equal to $1 - p(u_n)$ which converges to $1$ as $n \to \infty$. \\
	For the second term, we notice that each summand is bounded by
	$$| \Cov(\one\{X_0 > u_n\}, \one\{X_t > u_n\}) | \leq{} u_n^{-\alpha} (2-\theta_t), $$
	due to the the second part of  \Cref{lem:aux-bounds}. 
    From the first part of \Cref{lem:aux-bounds}, we know that $u_n^{-\alpha}/p(u_n)$ is bounded and  by assumption we know that $(2-\theta_t)$, $t \in \mathbb{Z}$, is summable. 
    Therefore, we can apply dominated convergence and obtain
	\begin{align*}
		& \lim_{n \to \infty} \frac 2 {p(u_n)} 
		\sum_{t=1}^{b_n-1} \left(1 - \frac{t}{b_n}\right) \Cov(\one\{X_0 > u_n\}, \one\{X_t > u_n\}) \\
		={}& 2 \sum_{t=1}^\infty \lim_{n \to \infty} \frac{\Cov(\one\{X_0 > u_n\}, \one\{X_t > u_n\})}{p(u_n)} \\
		={}& 2 \sum_{t=1}^\infty \lim_{n \to \infty} \big(\mathbb{P}(X_t > u_n \mid X_0 > u_n) - p(u_n)\big) = 2 \sum_{t=1}^\infty (2 - \theta_t),
	\end{align*}
	where we used that  $\mathbb{P}(X_t > u_n \mid X_0 > u_n) \to 2 - \theta_t$ as $n \to \infty$. Consequently, we obtain
	$$ \lim_{n \to \infty} \Var\left(\sqrt{b_n p(u_n)} \left(\frac{\singleest{}}{p(u_n)} - 1\right) \right) = 1 + 2 \sum_{t=1}^\infty (2-\theta_t). $$
	
	\noindent \textbf{Second Case:} $\sum_{t=1}^\infty (2 - \theta_t) = \infty$\\
	We can use the fact $\Cov(\one\{X_0 > u_n\}, \one\{X_t > u_n\}) \geq 0$ as the max-stable time series $(X_t)_{t \in \mathbb{Z}}$ is positively associated. Thus, for any fixed $k \in \mathbb{N}$ and sufficiently large $n$, we have
	\begin{align*}
		& \Var\left(\sqrt{b_n p(u_n)} 
		\frac{\singleest{}}{p(u_n)} \right) \\
		\geq{}& \frac {1}{p(u_n)} \Var(\one\{X_0 > u_n\})
		+ \frac 2 {p(u_n)} \sum_{t=1}^{k} \left(1 - \frac{t}{b_n}\right) \Cov(\one\{X_0 > u_n\}, \one\{X_t > u_n\}). 	
	\end{align*}
	Using the same calculations as above (here dominated convergence does not need to be applied as the sum consists of a fixed number of summands), we obtain
	$$ \lim_{n \to \infty} \Var\left(\sqrt{b_n p(u_n)} \frac{\singleest{}}{p(u_n)} \right) \geq 1 + 2 \sum_{t=1}^k (2-\theta_t). $$
	As this holds for any $k \in \mathbb{N}$, we finally get
	$$ \lim_{n \to \infty} \Var\left(\sqrt{b_n p(u_n)} \frac{\singleest{}}{p(u_n)} \right) 
	\geq 1 + 2 \sum_{t=1}^\infty (2-\theta_t) = \infty. $$ 
\end{proof}

Next, let us consider the numerator of our estimator \eqref{eq: estimator}.
To do so,  denote
\begin{align}
    \label{eq: kappa}
	\kappa_{t, T} = \lim_{u \to \infty} \frac{\mathbb{P}(u^{-1} (\Xvec{0},\ \Xvec{t}) \in C_{T}^2)}{\mathbb{P}(X_0 > u)}, 
	\quad
	t \in \mathbb{N}_0,
\end{align}
where $(X_t)_{t \in \mathbb{Z}}$ is a stationary max-stable time series and $C_{T}$  is defined in \eqref{eq: C-sets}. 
These quantities always exist as they can be expressed in terms of the exponent measure of the underlying time series.

\begin{prop} \label{prop:phun}
	Let $(X_t)_{t \in \mathbb{Z}}$	be a max-stable stationary time series
	with $\alpha$-Fr\'echet margins and pairwise extremal coefficients $(\theta_t)_{t \in \mathbb{Z}}$. Assume that $\sum_{t=1}^\infty (2-\theta_t) < \infty$ and choose a real sequence $u_n \to \infty$ as well as an integer sequence $b_n \to \infty$. 
	Then, we have
	$$ \lim_{n \to \infty} \Var\left( \sqrt{b_n p(u_n)} \frac{\multest{b_n} }{p(u_n)} \right) = \chi_{T}  + 2 \sum_{t=1}^\infty \kappa_{t, T}, $$
    where $\kappa_{t, T}$ is defined as in \eqref{eq: kappa}.
\end{prop}
\begin{proof}
	We have that 	
	\begin{align*}
		& \Var\left( \sqrt{b_n p(u_n)} \frac{\multest{b_n} }{p(u_n)} \right)\\
		={}&
		\frac 1 {b_n} \frac 1 {p(u_n)} \sum_{s=0}^{b_n} \sum_{t=0}^{b_n} \Cov(\one\{u_n^{-1} \Xvec{s}  \in C_{T}\}, \one\{u_n^{-1} \Xvec{}  \in C_{T}\}) \\
		={}& \frac {1}{p(u_n)} \Var(\one\{u_n^{-1} \Xvec{0}  \in C_{T}\}) \\ &+
		\frac 1 {b_n} \frac 2 {p(u_n)} \sum_{t=1}^{b_n - 1} (b_n - t) \Cov(\one\{u_n^{-1} \Xvec{0}  \in C_{T}\}, \one\{u_n^{-1} \Xvec{t}  \in C_{T}\}) \\
		={}& \frac {1} {p(u_n)} \Var(\one\{u_n^{-1} \Xvec{0}  \in C_{T}\})\\ & +
		\frac 2 {p(u_n)} \sum_{t=1}^{k} \left(1 - \frac t {b_n}\right) \Cov(\one\{u_n^{-1} \Xvec{0}  \in C_{T}\}, \one\{u_n^{-1} \Xvec{t}  \in C_{T}\})\\ &  +
		\frac 2 {p(u_n)} \sum_{t=k+1}^{b_n} \left(1 - \frac t {b_n}\right) \Cov(\one\{u_n^{-1} \Xvec{0}  \in C_{T}\}, \one\{u_n^{-1} \Xvec{t}  \in C_{T}\})
	\end{align*} 
	For fixed $k \in \mathbb{N}$, the sum of the first and the second term converges to
	$$ \chi_{T}  + 2 \sum_{t=1}^k \kappa_{t, T} $$
	(details similar to the proof of Proposition \ref{prop:pun}).
	From \Cref{lem:aux-bounds} it follows that that we can bound the third term by
	$K \cdot \sum_{t=k+1}^\infty (2 - \theta_t)$ for some constant $K > 0$.
	As this bound goes to zero as $k \to \infty$, taking the limit superior as $k \to \infty$ gives the desired result.    
\end{proof}

\begin{rem} \label{rem: well-definedness of asymptotic variances}
	\begin{enumerate}[(a)]
		\item Clearly, $\sum_{t=1}^\infty \kappa_{t, T} \leq \sum_{t=1}^\infty (2-\theta_t)$. 
		Hence, under the assumptions of \Cref{prop:phun}, the limiting variance of our estimator's numerator and denominator always exist (provided they are properly normalized).
		\item \label{rem: bias-negligible} If $u_n \rightarrow \infty$ sufficiently fast, we have that the bias of $\multest{} / p(u_n)$ is asymptotically negligible. 
		More specifically, if $b_np(u_n)^{3} \rightarrow 0$ as $n \rightarrow \infty$ we have that 
		\begin{align}
			\label{eq: var_PPhat}
			\lim_{n \to \infty} b_n p(u_n) \mathbb{E}\left[\left( \frac{\multest{b_n}}{p(u_n)} - \chi_T \right)^2\right] 
			= 
			\lim_{n \to \infty} \Var\left( \sqrt{b_n p(u_n)} \frac{\multest{b_n} }{p(u_n)} \right).
		\end{align}
		This follows from $\chi_T = \chi_{\{0\},\, T \setminus \{0\}}$ and \Cref{lem:aux-bounds}\eqref{lem: conv-bound-covariance} which implies that
		\begin{align*}
			\bigg\vert
			\sqrt{b_n p(u_n)} \bigg( 
			\mathbb{E} \bigg[\frac{\multest{}}{p(u_n)}\bigg] - \chi_{\{0\},\, T \setminus \{0\}}
			\bigg)
			\bigg\vert
			\leq \sqrt{b_n p(u_n)^3} \bigg( 1 + K  \frac{u_n^{-2\alpha}}{p(u_n)^2}\bigg)
		\end{align*}
		for some constant $K > 0$.
		The claim now follows from \Cref{lem:aux-limits}.
		\item From combining Prop.~\ref{prop:pun} and \ref{prop:phun} with Markov's inequality we get that 
      $$ \singleest{} / p(u_n) \pconv{} 1 \quad \text{and} \quad \multest{b_n}/p(u_n) \pconv{} \chi_{T}$$
      and, consequently, 
      $$ \multest{b_n} / \singleest{} \pconv{} \chi_{T} $$ 
       as $n \rightarrow \infty$ if $b_np(u_n) \rightarrow \infty$ as $n \rightarrow \infty$.
	\end{enumerate}
\end{rem}

\section{Asymptotic normality}
\label{sec: asym-normality-SRD}

As we've concluded in the last section, the properly normalized numerator and denominator of our estimator \eqref{eq: estimator} are weakly consistent under SRD.
In this section, we want to show that they are also asymptotically normally distributed.
Furthermore, we use this to prove that our ratio estimator is asymptocially normally distributed too.

\begin{thm} \label{thm: asym-normality-SRD}
	Let $(X_t)_{t \in \mathbb{Z}}$	be a max-stable stationary time series with $\alpha$-Fr\'echet margins and pairwise extremal coefficients $\theta_t$. 
	Assume that $\sum_{t=1}^\infty (2-\theta_t) < \infty$.
	Then, for any sequence $u_n \to \infty$ such that $n p(u_n) \rightarrow \infty$ as $n \rightarrow \infty$, we have
	\begin{align*}
		\sqrt{n p(u_n)} \bigg( \frac{\singleest{n}}{p(u_n)}  - 1 \bigg) &\dconv{} \mathcal{N}(0, \sigma_0^2) \\[2mm]
		\sqrt{n p(u_n)} \bigg( \frac{\multest{n}}{p(u_n)}  - \chi_{T} \bigg)  &\dconv{} \mathcal{N}(0, \sigma_{T}^2),
	\end{align*}
	as $n \rightarrow \infty$, where $\sigma_0^2$ and $\sigma_{T}^2$ are the limiting variances as defined in Propositions \ref{prop:pun} and \ref{prop:phun}.
\end{thm}

\begin{proof}
	For legibility this proof has been moved to \Cref{sec: proof-asym-normality-SRD}.
\end{proof}

Finally, let us combine the results from \Cref{thm: asym-normality-SRD} to show that our quotient estimator $\widehat \chi_{T ,n}$ is asymptotically normal.

\begin{thm} \label{thm: asym-normal-chi-SRD}
	Under the assumptions of \Cref{thm: asym-normality-SRD} it holds that
	\begin{align*}
		\sqrt{n p(u_n)} \bigg(\widehat \chi_{T ,n} - \frac{\mathbb{P}(u_n^{-1} \Xvec{} \in  C_{T})}{\mathbb{P}(u_n^{-1} \Xvec{} \in  C_0)} \bigg)
		\dconv
		\mathcal{N}(0, \sigma^2)
	\end{align*}
	as $n \rightarrow \infty$ for some $\sigma^2 > 0$.
	Additionally, if $np(u_n)^3 \rightarrow 0$ as $n \rightarrow \infty$, it holds that
	\begin{align*}
		\sqrt{n p(u_n)} \big(\widehat \chi_{T ,n} - \chi_{T ,n} \big)
		\dconv
		\mathcal{N}(0, \sigma^2)
	\end{align*}
	as $n \rightarrow \infty$.
\end{thm} 

\begin{rem}
    The novelty of this result is the weak assumption $\sum_{t=1}^\infty (2-\theta_t) < \infty$ on the dependence structure of the time series.   
    Typically, asymptotic normality is proven under the assumption of strong mixing properties of the time series. 
    For example, \cite{davismikosch09} proves asymptotic normality of a more general version of our estimator \eqref{eq: estimator} under the assumption that $X$ is, among other assumptions, strongly mixing with mixing rate $(\alpha_t)_{t \in \mathbf{Z}}$ that decays sufficiently fast such that
    \begin{align*}
	   \lim_{n \rightarrow \infty} m_n \sum_{t = r_n}^\infty \alpha_t = 0,
    \end{align*}
    where $m_n, r_n \rightarrow \infty$ with $m_n / n \rightarrow 0$ and $r_n / m_n \rightarrow 0$ as $n \rightarrow \infty$.
    Compared to this, our assumption is much simpler and also weaker as these mixing assumptions in particular imply that the limiting variance $\sigma_0^2$ in \Cref{thm: asym-normality-SRD} is finite.
\end{rem}

\begin{proof}
	First, let us show the following multivariate CLT:
    For
	\begin{align*}
		S_n := (n p(u_n))^{-1/2} \sum_{t = 1}^{n} \begin{pmatrix}
			\one\{u_n^{-1} \Xvec{}  \in C_{T}\} -
			\mathbb{P}(u_n^{-1} \Xvec{} \in  C_{T}) \\
			\one\{u_n^{-1} \Xvec{}  \in C_{0}\} -
			\mathbb{P}(u_n^{-1} \Xvec{} \in  C_{0})
		\end{pmatrix}
	\end{align*}
    it holds that
    \begin{align}
        \label{eq: multivar-CLT-SRD}
        S_n \dconv \mathcal{N}(0, \Sigma),
    \end{align}
	where $\Sigma$ is a covariance matrix. 
	In order to prove \eqref{eq: multivar-CLT-SRD}, it suffices to show $z^{\prime} S_n \dconv \mathcal{N}(0, z^{\prime} \Sigma z)$ for all $z = (z_1,\, z_2)^\prime \in \mathbb{R}^2$.
	
	We can prove this using the same large/small block technique that we used to prove \Cref{thm: asym-normality-SRD}. 
	Hence, the calculations are the same as in \Cref{sec: proof-asym-normality-SRD}.
	Note that the variance of the small blocks are negligible again because $\Var(X + Y) \leq \Var(X) + \Var(Y) + 2 \sqrt{\Var(X) \Var(Y)}$.
	Furthermore, the Lipschitz constant of the functions in the analogue of \eqref{eq: covariance-rewritten} will be bound by $\vert z_1 \vert + \vert z_2 \vert$. 
	This does not affect the remainder of the proof.
	
	Thus, we can consider iid.~copies of 
	\begin{align*}
		T_n(z) 
		= 
		(n p(u_n))^{-1 / 2} \sum_{t = 1}^{\blarge}
		\Big[ 
		&z_1 \Big(\one\{u_n^{-1} \Xvec{}  \in C_{T}\} -
		\mathbb{P}(u_n^{-1} \Xvec{} \in  C_{T}) 
		\Big) \\
		+ 
		&z_2 \Big( 
		\one\{u_n^{-1} \Xvec{}  \in C_{0}\} -
		\mathbb{P}(u_n^{-1} \Xvec{} \in  C_{0})
		\Big)
		\Big].
	\end{align*}
	Now, we have to verify Lindeberg's condition for all $\varepsilon > 0$.
	Similar to our proof in \Cref{sec: proof-asym-normality-SRD}, Chebychev's inequality yields that $k_n \mathbb{E} \big[
		T_n^2(z) \one \big\{ \vert T_n(z) \vert > \varepsilon \big\}
		\big]$ 
  is bounded from above by
	\begin{align*}
		\widetilde{K} \frac{k_n}{n p(u_n)} \blarge^2 \mathbb{P}\big(\vert T_n(z) \vert > \varepsilon\big) 
		\leq
  \widetilde{K} 
		\frac{\blarge^2}{n p(u_n)} \frac{k_n \Var(T_n(z))}{\varepsilon^2}.
	\end{align*}
	for some constant $\widetilde{K} > 0$.
	
	By the same assumption on the block length $\blarge{}$ that we used in \Cref{sec: proof-asym-normality-SRD}, $\blarge^2 / (n p(u_n)) \rightarrow 0$ as $n \rightarrow \infty$.
	Now, the CLT \eqref{eq: multivar-CLT-SRD} follows from $k_n \Var(T_n(z)) \rightarrow z^{\prime} \Sigma z$ as $n \rightarrow \infty$ due to \Cref{thm: asym-normality-SRD}.
	
	Next, let us show that 
	\begin{align}
		\label{eq: CLT-chi-prob-SRD}
		\sqrt{n p(u_n)} \bigg(\widehat \chi_{T ,n}
		- 
		\frac{\mathbb{P}(u_n^{-1} \Xvec{} \in  C_{T})}{p(u_n)}  
		\bigg)
		\dconv \mathcal{N}(0, F^{\prime} \Sigma F),
	\end{align}
	where $F = (1,\ -\chi_{T})^\prime$.
	Let us denote $P_n( C_{T}) =\mathbb{P}(u_n^{-1} \Xvec{} \in  C_{T})$. 
	Then, we can rewrite
	\begin{align*}
		&\sqrt{n p(u_n)} \bigg(\widehat \chi_{T ,n}
		- 
		\frac{
			\mathbb{P}(u_n^{-1} \Xvec{} \in  C_{T})}{p(u_n)}  
		\bigg) \\[2mm]
		=
		&\sqrt{n p(u_n)} \frac{p(u_n)}{\singleest{n}} \bigg( 
		\frac{
			\big[\multest{n} - P_n(C_{T}) \big] p(u_n)
			-
			\big[\singleest{n} - p(u_n) \big] P_n(C_{T})
		}{p(u_n)^2}
		\bigg) \\[2mm]
		=
		&\sqrt{n p(u_n)}  \frac{p(u_n)}{\singleest{n}}
		\bigg( 
		1,\,\ -\frac{ P_n(C_{T})}{p(u_n)}
		\bigg)
		\frac{1}{p(u_n)} \begin{pmatrix}
			\multest{n} - P_n(C_{T}) \\
			\singleest{n} - p(u_n)
		\end{pmatrix} \\[2mm]
		= &\frac{p(u_n)}{\singleest{n}}
		\bigg( 
		1,\,\ -\frac{ P_n(C_{T})}{p(u_n)}
		\bigg)
		\sqrt{\frac{n}{p(u_n)}} \begin{pmatrix}
			\multest{n} - P_n(C_{T}) \\
			\singleest{n} - p(u_n)
		\end{pmatrix}
	\end{align*}
	Then, CLT \eqref{eq: CLT-chi-prob-SRD} follows from Slutsky's theorem and CLT \eqref{eq: multivar-CLT-SRD}.
	Finally, from \Cref{rem: well-definedness of asymptotic variances}\eqref{rem: bias-negligible} we know that
	\begin{align*}
		\sqrt{n p(u_n)} \bigg( 
		\chi_{T} 
		- 
		\frac{\mathbb{P}(u_n^{-1} \Xvec{} \in  C_{T})}{p(u_n)} 
		\bigg)
		= \sqrt{
			n p(u_n) \bigg( 
			\chi_{T} 
			- 
			\mathbb{E}\bigg[\frac{\multest{n}}{p(u_n)} \bigg] 
			\bigg)^2
		}
	\end{align*}
	converges to zero if $n p(u_n)^3 \rightarrow 0$ as $n \rightarrow \infty$. 
	This completes our proof.
\end{proof}

\section{Long Range Dependence as Phase Transition}
\label{sec: conjectures-LRD}

We have seen in the previous sections that our main results, Theorems \ref{thm: asym-normality-SRD} and \ref{thm: asym-normal-chi-SRD}, hold if $\sum_{t = 1}^\infty (2 - \theta_t) < \infty$.
Interestingly, this exact condition is indicative of a max-stable time series being long or short range dependent in the sense of \Cref{def: LRD-KS}.
More precisely, \cite[Theorem 4.3]{makogin2021long} shows that a max-stable time series is LRD in the sense of \Cref{def: LRD-KS} iff $\sum_{t = 1}^\infty (2 - \theta_t) = \infty$.

In this section, we show that the limiting behavior of $\singleest{}$ undergoes a phase transition as the underlying time series becomes LRD.
Moreover, we close this section with an additional discussion of conjectures on how to extend our results even further.

\Cref{prop:pun} shows that the asymptotic variance of the numerator is infinite if $\sum_{t = 1}^\infty (2 - \theta_t) = \infty$ and the typical normalization $\sqrt{n p(u_n)}$ is used.
This begs the question whether a different normalization of the denominator can yield a finite limiting variance in the case where $\sum_{t = 1}^\infty (2 - \theta_t) < \infty$.
And this is indeed the case.

\begin{prop}
	\label{prop: lim-var-pun-LRD}
	Let $(X_t)_{t \in \mathbb{Z}}$	be a max-stable stationary time series with $\alpha$-Fr\'echet margins and pairwise extremal coefficients $(\theta_t)_{t \in \mathbb{Z}}$.
	Furthermore, assume that $2 - \theta_t = C t^{-\delta}$, $t \in \mathbb{Z}$, for some constants $C > 0$ and $\delta \in (0, 1)$.
	Then, for any real sequence $u_n \to \infty$ and any integer sequence $b_n \to \infty$, we have
	\begin{align}
		\label{eq: LRD-variances}
		\lim_{n \to \infty} \Var\left(\sqrt{b_n^\delta p(u_n)} \frac{\singleest{}}{p(u_n)} \right) 
		= 
		\frac{2C}{(1- \delta)(2 - \delta)}
	\end{align}
	
\end{prop}

\begin{proof}

	Since $\delta < 1$, the first and second term converge to zero as $n \rightarrow \infty$.
	Due to \Cref{lem:aux-bounds}, we know for $t > 0$ that
	\begin{align*}
		(2 - \theta_t) \bigg(1 - \frac{5u_n^{-2\alpha}}{2p(u_n)} \bigg) 
		\leq
		\frac{\Cov(\one \{X_0 > u_n \}, \one \{X_t > u_n \})}{p(u_n)}
		\leq
		(2 - \theta_t) \bigg(1 + \frac{5u_n^{-2\alpha}}{2p(u_n)} \bigg)
	\end{align*}
	By \Cref{lem:aux-bounds} we know that $u_n^{-2\alpha} / p(u_n) \rightarrow 0$ as $n \rightarrow \infty$.
    Thus,
	\begin{align*}
		&\lim_{n \rightarrow \infty} \Var\Big(\sqrt{n^\delta p(u_n)} \frac{\singleest{}}{p(u_n)} \Big) \\
		&= 
		\lim_{n \rightarrow \infty} 2 b_n^{\delta - 1} \sum_{t = 1}^{b_n - 1} \bigg(1 - \frac{t}{b_n}\bigg)  \frac{\Cov(\one \{X_0 > u_n \}, \one \{X_t > u_n \})}{p(u_n)} \\
		&=
		\lim_{n \rightarrow \infty} 2b_n^{\delta - 1} \sum_{t = 1}^{b_n - 1} \bigg(1 - \frac{t}{b_n}\bigg) (2 - \theta_t) \\
		&= 
		2C \lim_{n \rightarrow \infty} \bigg[
		\sum_{t = 1}^{b_n - 1} 
		\frac{1}{b_n}
		\bigg( \frac{t}{b_n} \bigg)^{-\delta}
		- 
		\sum_{t = 1}^{b_n - 1} 
		\frac{1}{b_n}
		\bigg( \frac{t}{b_n} \bigg)^{-\delta + 1}
		\bigg]\\
		&=
		2C \bigg[\int_0^{1} x^{-\delta} \mathrm{d}x
		-
		\int_0^{1} x^{1-\delta} \mathrm{d}x \bigg]
		= \frac{2C}{(1 - \delta) (2 - \delta)}
	\end{align*}
\end{proof}

\begin{rem}
	\label{rem: lim-var-phun-LRD}
	Clearly, the convergence rate of the variance of $\multest{}$ is bounded by the convergence rate of the variance of $\singleest{}$.
	Hence, there exists a constant $C \in [0, \infty)$ such that
	\begin{align*}
		\limsup_{n \to \infty} \Var\left(\sqrt{b_n^\delta p(u_n)} \frac{\multest{}}{p(u_n)} \right) = C.
	\end{align*}
	Thus, the properly normalized variance of $\multest{}$ is bounded by a finite constant, possibly zero.
\end{rem}

As for our ratio estimator $\widehat \chi_{T , b_n}$, we conjecture that its limiting variance uses the same rate of convergence as the denominator. 
This could possibly shown with a Delta-method argument if assumptions are appropriately chosen such that remainder terms of the resulting Taylor expansion are negligible.

\begin{conj}
	\label{conj: asymp variance of ratio}
	Let $(X_t)_{t \in \mathbb{Z}}$	be a max-stable stationary time series
	with $\alpha$-Fr\'echet margins and pairwise extremal coefficients $(\theta_t)_{t \in \mathbb{Z}}$. 
	Furthermore, assume that $2 - \theta_t = C t^{-\delta}$, $t \in \mathbb{Z}$, for some constants $C > 0$ and $\delta \in (0, 1)$.
	Then, under appropriate assumptions on the asymptotic covariance of $\singleest{n}$ and $\multest{n}$ there exists a constant $K > 0$ such that for any real sequence $u_n \to \infty$ and any integer sequence $b_n \to \infty$, we have
	\begin{align*}
		\lim_{n \rightarrow \infty} \Var \bigg(
		\sqrt{n^\delta p(u_n)} \frac{\multest{n}}{\singleest{n}}
		\bigg) = K.
	\end{align*}
\end{conj}

Finally, in light of \Cref{conj: asymp variance of ratio}, we conjecture that if $2 - \theta_t = C t^{-\delta}$, $t \in \mathbb{Z}$, for some constants $C > 0$ and $\delta \in (0, 1)$, we conjecture an  LRD-analogue of our main  result, \Cref{thm: asym-normal-chi-SRD}, i.e. the limiting distribution of 
$$
\sqrt{n^\delta p(u_n)} \bigg(\frac{\multest{n}}{ \singleest{n}}- \chi_T\bigg)
$$ 
is non-degenerate as $n \rightarrow \infty$.
\bibliographystyle{imsart-number} 
\bibliography{lit}       

\appendix

\section{Proof of \Cref{thm: asym-normality-SRD}}
\label{sec: proof-asym-normality-SRD}

In this section, we are going to prove \Cref{thm: asym-normality-SRD}.
To do so, let us first show the following two lemmas.

\begin{lemma} \label{lem:aux-limits}
	Let $(X_t)_{t \in \mathbb{Z}}$	be a max-stable stationary time series
	with $\alpha$-Fr\'echet margins and extremal coefficients $\theta_T$, $T \subset \ZZ$ finite.
	Then, it holds that
	\begin{enumerate}[(a)]
		\item $\lim_{u \to \infty} u^{-\alpha} / \PP(X_0 > u) = 1$.
		\item For every pair of disjoint finite subsets $A,B \subset \mathbb{Z}$, the limit
		\begin{align*}
			\chi_{A,B} := \lim_{u \to \infty} \frac{\Cov(\one\{X_s > u \, \forall s \in A\}, \one\{X_t > u \, \forall t \in B\})}{\PP(X_0 > u)}
		\end{align*}
		exists and is given by 
		\begin{align*}   
			\chi_{A,B} = \sum_{\substack{C \subset A \cup B\\ C \cap A \neq \emptyset, C \cap B \neq \emptyset}} (-1)^{|C|}
			\left( \theta_{C \cap A} + \theta_{C \cap B} - \theta_C \right).
		\end{align*}
	\end{enumerate}
\end{lemma}
\begin{proof}
	\begin{itemize}
		\item[1.] Follows directly from $\PP(X_0>u) = 1 - \exp(-u^{-\alpha})$ and a Taylor expansion of the function $x \mapsto \exp(-x)$ around $0$.
		\item[2.] By the inclusion-exclusion principle, we obtain
		\begin{align*}
			&	\Cov(\one\{X_s > u, \, \forall s \in A\}, \one\{X_t > u, \, \forall t \in B\})  \\
			={}& \mathbb{P}(X_r > u \, \forall r \in A \cup B) 
			- \mathbb{P}(X_s > u \, \forall s \in A) \mathbb{P}(X_t > u \, \forall t \in A) \\
			={}& \Big(1 - \mathbb{P}(X_r \leq u \text{ for some } r \in A \cup B)\Big) \\
			&        - \Big(1 - \mathbb{P}(X_s \leq u \text{ for some } s \in A)\Big)
			\Big(1 - \mathbb{P}(X_t \leq u_n \text{ for some } t \in B)\Big) \\
			={}& 1 + \sum_{\emptyset \neq C' \subset A \cup B} (-1)^{|C'|} \belowprob{C^\prime} \\
			& - 
   \bigg(1 + \sum_{\emptyset \neq A' \subset A} (-1)^{|A'|} \belowprob{A^\prime}\bigg)
			\bigg(1 + \sum_{\emptyset \neq B' \subset B} (-1)^{|B'|} \belowprob{B^\prime})\bigg),
		\end{align*}
        where $\belowprob{M} := \mathbb{P}(X_t \leq u \text{ for all } t \in M)$, $M \subset \mathbb{Z}$.
		Expanding the product on the right-hand side of the equation, one can see that all the summands with $C \subset A$ and $C \subset B$ cancel out and the expression simplifies to 
		\begin{align}
			&	\Cov(\one\{X_s > u, \, \forall s \in A\}, \one\{X_t > u, \, \forall t \in B\}) \nonumber \\
			={}& \sum_{\substack{C \subset A \cup B\\ C \cap A \neq \emptyset, \\ C \cap B \neq \emptyset}} \Big[ (-1)^{|C|} \belowprob{C} 
            - 
            (-1)^{|C \cap A| + |C \cap B|} \belowprob{C \cap A} \belowprob{C \cap B} \Big] \nonumber \\
			={}&   \sum_{\substack{C \subset A \cup B\\ C \cap A \neq \emptyset, \\ C \cap B \neq \emptyset}} (-1)^{|C|} \Big( \exp(- u^{-\alpha} \theta_C) -  \exp(- u^{-\alpha} \theta_{C \cap A}) \exp(- u^{-\alpha} \theta_{C \cap B}) \Big). \label{eq:cov-2-sumsets}
		\end{align}
		By applying the Taylor expansion of the function $x \mapsto \exp(-x)$ around $0$ for each exponential term, we finally obtain
		\begin{align*}
			&	\Cov(\one\{X_s > u, \, \forall s \in A\}, \one\{X_t > u, \, \forall t \in B\}) \\
			={}&  u^{-\alpha} \sum_{\substack{C \subset A \cup B\\ C \cap A \neq \emptyset, C \cap B \neq \emptyset}}   (-1)^{|C|} \left( \theta_{C \cap A} + \theta_{C \cap B} - \theta_C \right)  + \mathcal{O}(u^{-2\alpha})
		\end{align*}
		Now, the assertion follows from the first part of the lemma.
	\end{itemize}
\end{proof}

The following lemma will give us two crucial auxiliary results to prove asymptotic normality. 
The first result bounds the covariance of indicators. 
The second result gives a bound on the covariance's convergence rate.

\begin{lemma} \label{lem:aux-bounds}
	Let $(X_t)_{t \in \mathbb{Z}}$	be a max-stable stationary time series
	with $\alpha$-Fr\'echet margins and extremal coefficients $\theta_T$, $T \subset \ZZ$ finite, where we denote the pairwise extremal coefficients by $\theta_h := \theta_{\{0,h\}}$, $h>0$, for short.
	Furthermore, let $A,B \subset \mathbb{Z}$ be finite and disjoint.
	Then, the following bounds hold.
	\begin{enumerate}[(a)]
		\item The covariance of indicators of exceedances decays at rate $u^{-\alpha}$. More precisely,
		\begin{align*}
			&|\Cov(\one\{X_s > u \, \forall s \in A\}, \one\{X_t > u \, \forall t \in B\})| \\
			&\leq{} 
			u^{-\alpha} 2^{|A| + |B| -2} \sum_{s \in A} \sum_{t \in B} (2-\theta_{|s-t|})
		\end{align*}
		\item \label{lem: conv-bound-covariance}The difference between the covariance of indicators of exceedances and the asymptotically equivalent term $\PP(X_0 > u) \chi_{A,B}$ decays at rate $u^{-2\alpha}$.
		More precisely,
		\begin{align*}
			& |\Cov(\one\{X_s > u \, \forall s \in A\}, \one\{X_t > u \, \forall t \in B\})  -  \PP(X_0 > u) \chi_{A,B} | \\
			\leq{}& u^{-2\alpha} \left(|A|+|B|+ \frac 1 2\right) 2^{|A| + |B| - 2} \sum_{s \in A} \sum_{t \in B} (2-\theta_{|s-t|})
		\end{align*}
	\end{enumerate}
\end{lemma}
\begin{proof}
		We begin to prove the first part. Equation \eqref{eq:cov-2-sumsets} yields
		\begin{align*}
			& |\Cov(\one\{X_s > u, \, \forall s \in A\}, \one\{X_t > u, \, \forall t \in B\})| \\ 
			={}&  \Big| \sum_{\substack{C \subset A \cup B\\ C \cap A \neq \emptyset, C \cap B \neq \emptyset}} (-1)^{|C|} \Big[ \exp(-u^{-\alpha} \theta_C) -  \exp(-u^{-\alpha} \theta_{C \cap A}) \exp(-u^{-\alpha} \theta_{C \cap B}) \Big] \Big| \\
			\leq{}&  \sum_{\substack{C \subset A \cup B\\ C \cap A \neq \emptyset, C \cap B \neq \emptyset}} \Big| \exp(-u^{-\alpha} \theta_C) -  \exp(-u^{-\alpha} \theta_{C \cap A}) \exp(-u^{-\alpha} \theta_{C \cap B}) \Big| \\
			={}& \sum_{\substack{C \subset A \cup B\\ C \cap A \neq \emptyset, C \cap B \neq \emptyset}}  \Big| \int_{u^{-\alpha} \theta_C}^{u^{-\alpha} (\theta_{C \cap A} + \theta_{C \cap B}) } \exp(-x) \, \mathrm{d}x \Big| \\
			\leq{}& \sum_{\substack{C \subset A \cup B\\ C \cap A \neq \emptyset, C \cap B \neq \emptyset}} \int_{u^{-\alpha} \theta_C}^{u^{-\alpha} (\theta_{C \cap A} + \theta_{C \cap B}) } 1 \, \mathrm{d}x \\
			={}& \sum_{\substack{C \subset A \cup B\\ C \cap A \neq \emptyset, C \cap B \neq \emptyset}} u^{-\alpha} \left( \theta_{C \cap A} + \theta_{C \cap B} - \theta_C \right).
		\end{align*}
	
		Now, we notice that, for arbitrary sets $A_1, A_2, A_3 \subset \mathbb{Z}$, we have that, by Equation (15) in \cite{schlather-tawn-02},
		\begin{align} \label{eq: tawn-bound}
			0 \leq{}& \theta_{A_1 \cup A_2} + \theta_{A_3} - \theta_{A_1 \cup A_2 \cup A_3}\\
			\leq{}& \theta_{A_1 \cup A_2} + \theta_{A_3} - \left( \theta_{A_1 \cup A_2} +  \theta_{A_1 \cup A_3} +  \theta_{A_2 \cup A_3} - \theta_{A_1} - \theta_{A_2} - \theta_{A_3}  \right) \notag \\
			={}& (\theta_{A_1} + \theta_{A_3} - \theta_{A_1 \cup A_3}) +  (\theta_{A_2} + \theta_{A_3} - \theta_{A_2 \cup A_3}). \notag
		\end{align}
		Applying this bound iteratively to sets of the type $A_1 = \{s\}$, $A_2 = C \cap A \setminus \{s\}$, $A_3 = C \cap B$ for $s \in A \cap C$ and $A_1 = \{t\}$, $A_2 = C \cap B \setminus \{t\}$, $A_3 = C \cap A$ for $t \in B \cap C$, respectively we can decompose the sets $C \cap A$ and $C \cap B$ into singletons and, using that $\theta_{\{t\}} = 1$ for all $t \in \mathbb{Z}$, we eventually obtain
		\begin{align} 
			\theta_{C \cap A} + \theta_{C \cap B} - \theta_C 
			&\leq 
			\sum_{s \in A \cap C} \sum_{t \in B \cap C} (\theta_{\{s\}} + \theta_{\{t\}} - \theta_{\{s,t\}}) \notag \\
			&= \sum_{s \in A \cap C} \sum_{t \in B \cap C} (2 - \theta_{|s-t|}).
			\label{eq:bound-mv.ec-pw.ec}
		\end{align}
		Thus, we obtain
		\begin{align*}
			|\Cov(\one\{X_s > u \, \forall s \in A\}, \one\{X_t > u \, \forall t \in B\}) | \\
			\leq{} 
			u^{-\alpha}  \sum_{\substack{C \subset A \cup B\\ C \cap A \neq \emptyset, C \cap B \neq \emptyset}} \sum_{s \in A \cap C} \sum_{t \in B \cap C} (2 - \theta_{|s-t|}). 
		\end{align*}
		The assertion follows from the fact that, for each $(s,t) \in A \times B$, there exist
		$2^{|A|-1}$ subsets of $A$ that contain $s$ and $2^{|B|-1}$ subsets of $B$ that contain $t$.
		
		Next, let us prove the second part. 
		Plugging in Equation \eqref{eq:cov-2-sumsets} and the expression for $\chi_{A,B}$
		from the second part of  \Cref{lem:aux-limits}, we obtain 
		\begin{align}
			& |\Cov(\one\{X_s > u \, \forall s \in A\}, \one\{X_t > u \, \forall t \in B\})  -  \PP(X_0 > u) \chi_{A,B} | \notag \\
			={}& \bigg| \sum_{\substack{C \subset A \cup B\\ C \cap A \neq \emptyset, C \cap B \neq \emptyset}} (-1)^{|C|} \Big[ \exp(-u^{-\alpha} \theta_C) -  \exp(-u^{-\alpha} \theta_{C \cap A}) \exp(-u^{-\alpha} \theta_{C \cap B}) \Big] \notag \\
			& \qquad - \PP(X_0 > u_n) \sum_{\substack{C \subset A \cup B\\ C \cap A \neq \emptyset, C \cap B \neq \emptyset}} (-1)^{|C|}
			\left( \theta_{C \cap A} + \theta_{C \cap B} - \theta_C \right) \bigg| \notag \\
			\leq{}& \sum_{\substack{C \subset A \cup B\\ C \cap A \neq \emptyset,  C \cap B \neq \emptyset}} \Big| \exp(-u^{-\alpha} \theta_C) -  \exp(-u^{-\alpha} (\theta_{C \cap A} + \theta_{C \cap B}))  \notag \\
			& \qquad \qquad \qquad \qquad
			- \PP(X_0 > u) \left( \theta_{C \cap A} + \theta_{C \cap B} - \theta_C \right) \Big|.    
            \label{eq: first_step covariance conv rate}
		\end{align}
		From $x - x^2/2 \leq 1 - e^{-x} \leq x$ for all $x >0$, it follows that $u^{-\alpha} - u^{-2\alpha} / 2 \leq \mathbb{P}(X_0 > u) \leq u^{-\alpha}$.
		Consequently, we can bound
		\begin{align*}
			& |\Cov(\one\{X_s > u \, \forall s \in A\}, \one\{X_t > u \, \forall t \in B\})  -  \PP(X_0 > u) \chi_{A,B} | \\
			\leq{}& \max \Big\{ |\Cov(\one\{X_s > u \, \forall s \in A\}, \one\{X_t > u \, \forall t \in B\})  -  (u^{-\alpha} - u^{-2\alpha} / 2) \chi_{A,B} |, \\
			& \qquad |\Cov(\one\{X_s > u \, \forall s \in A\}, \one\{X_t > u \, \forall t \in B\})  -  u^{-\alpha} \chi_{A,B} | \Big\}
        \end{align*}
        
        Then, applying analogous calculations as in \eqref{eq: first_step covariance conv rate} to both parts of the maximum yields
        \begin{align*}
            &|\Cov(\one\{X_s > u \, \forall s \in A\}, \one\{X_t > u \, \forall t \in B\})  -  \PP(X_0 > u) \chi_{A,B} | \\
            \leq{}& \sum_{\substack{C \subset A \cup B\\ C \cap A \neq \emptyset, \\ C \cap B \neq \emptyset}} \Big| \exp(- u^{-\alpha} \theta_C) -  \exp(- u^{-\alpha} (\theta_{C \cap A} + \theta_{C \cap B})) - u^{-\alpha} \left( \theta_{C \cap A} + \theta_{C \cap B} - \theta_C \right) \Big| \\
			& \qquad + \frac 1 2  \sum_{\substack{C \subset A \cup B\\ C \cap A \neq \emptyset, C \cap B \neq \emptyset}} \left( \theta_{C \cap A} + \theta_{C \cap B} - \theta_C \right) u^{-2\alpha} \\
			={}& \sum_{\substack{C \subset A \cup B\\ C \cap A \neq \emptyset, C \cap B \neq \emptyset}} 
			\left[ \int_{u^{-\alpha} \theta_C}^{u^{-\alpha} (\theta_{C \cap A} + \theta_{C \cap B})} (1 - e^{-x}) \, \mathrm{d} x + \frac{\theta_{C \cap A} + \theta_{C \cap B} - \theta_C }{2} u^{-2\alpha} \right] \\
			\leq{}& \sum_{\substack{C \subset A \cup B\\ C \cap A \neq \emptyset, C \cap B \neq \emptyset}} 
			\bigg[ \int_{u^{-\alpha} \theta_C}^{u^{-\alpha} (\theta_{C \cap A} + \theta_{C \cap B})} x \, \mathrm{d} x + \frac {\theta_{C \cap A} + \theta_{C \cap B} - \theta_C}{2} u^{-2\alpha} \bigg]  \\
			={}& \frac 1 2 u^{-2\alpha} \sum_{\substack{C \subset A \cup B\\ C \cap A \neq \emptyset, C \cap B \neq \emptyset}} \left[ (\theta_{C \cap A} + \theta_{C \cap B})^2 - \theta_C^2 + \left( \theta_{C \cap A} + \theta_{C \cap B} - \theta_C \right)   \right] \\
			={}& \frac 1 2 u^{-2\alpha} \sum_{\substack{C \subset A \cup B\\ C \cap A \neq \emptyset, C \cap B \neq \emptyset}} \left[ (\theta_{C \cap A} + \theta_{C \cap B} + \theta_C + 1) \left( \theta_{C \cap A} + \theta_{C \cap B} - \theta_C \right)   \right] \\
			\leq{}& u^{-2\alpha} \left(|A|+|B|+ \frac 1 2\right) \sum_{\substack{C \subset A \cup B\\ C \cap A \neq \emptyset, C \cap B \neq \emptyset}} \left( \theta_{C \cap A} + \theta_{C \cap B} - \theta_C \right),
		\end{align*}
		where we used $ \theta_{C \cap A} + \theta_{C \cap B} - \theta_C =  \theta_{C \cap A} + \theta_{C \cap B} - \theta_{C \cap (A \cup B)} \geq 0$ by inequality \eqref{eq: tawn-bound} and $\theta_M \leq \vert M\vert$ for all finite sets $M \subset \mathbb{Z}$. 
        The assertion follows by employing Equation \eqref{eq:bound-mv.ec-pw.ec} analogously to the proof of the first part of the lemma.   
\end{proof}

This proof follows the same large/small block argument as given in the proof of Theorem 3.2 in \cite{davismikosch09}.
We prove only the limit theorem for $\multest{}$. 
The result for $\singleest{}$ follows analogously.
Let us start by denoting $C := C_T$, $p_0 := \mathbb{P}(u_n^{-1} \Xvec{0} \in C )$ and 
\begin{align*}
	Y_{n, t} := \big(np(u_n)\big)^{-1/2} \Big( \one\{ u_n^{-1} \Xvec{} \in C\} - p_0 \Big), \quad t = 1, \ldots, n.
\end{align*}
Now, we want to consider $k_n$ large blocks of size $\blarge$. To this end, choose $\blarge \in \mathbb{N}$ such that $k_n = n / \blarge$ is a positive integer and $\blarge \to \infty$ as $n \to \infty$. Moreover, for technical reasons that will become clear in the last part of this proof, we assume that $\blarge^2 / (np(u_n)) \rightarrow 0$ as $n \to \infty$. Note that this can always be achieved by letting $\blarge$ grow sufficiently slowly. In addition to $\blarge$ we choose small block sizes $\bsmall \in \mathbb{N}$ such that $\bsmall \to \infty$, but $\bsmall/\blarge \to 0$ as $n \to \infty$.
Using these block sizes we can define blocks
\begin{align*}
	I_{n, i} &:= \big\{ (i - 1) \blarge + 1, \ldots, i \blarge \big\} \\
	\widetilde{I}_{n, i} &:= \big\{ (i - 1) \blarge +  \bsmall + 1, \ldots, i \blarge \big\} \\
	J_{n, i} &:= I_{n, i} \setminus \widetilde{I}_{n, i} = \big\{ (i - 1) \blarge + 1, \ldots, (i-1) \blarge + \bsmall \big\}
\end{align*}
for $i = 1, \ldots, k_n$.
For any index sex $B \subset \mathbb{N}_{0}$ we write $S_n(B) = \sum_{t \in B} Y_{n, t}$.

Next, let us show that the small blocks $J_{n, i}$, $i = 1, \ldots, k_n$, are asymptotically negligible, i.e.
\begin{align}
	\label{eq: small-blocks-negligible}
	\Var \Big( \sum_{i = 1}^{k_n} S_n(J_{n, i}) \Big) \rightarrow 0
\end{align}
as $n \rightarrow \infty$.
Notice that
\begin{align*}
	\Var \Big( \sum_{i = 1}^{k_n} S_n(J_{n, i}) \Big)
	&\leq 
	k_n \Var(S_n(J_{n, 1})) + 2 \sum_{i = 1}^{k_n} \sum_{j = 1}^{i - 1} \Cov \big( S_n(J_{n, i}), S_n(J_{n, j}) \big) \\
	&=: P_1 + P_2.
\end{align*}
Thus, let us show that both $P_1$ and $P_2$ converge to zero as $n \rightarrow \infty$. 
By \Cref{prop:phun} and $\bsmall / \blarge \rightarrow 0$ as $n \rightarrow \infty$, it follows that that
\begin{align*}
	P_1 = \frac{\bsmall}{\blarge} \Var \bigg(
	\sqrt{\bsmall p(u_n)} \frac{\multest{\bsmall}}{p(u_n)}
	\bigg)
	\rightarrow 0
\end{align*}
as $n \rightarrow \infty$.
Furthermore, we get 
\begin{align*}
	P_2 
	&=
	2 \sum_{i = 1}^{k_n} \sum_{j = 1}^{i - 1} \Cov \big( S_n(J_{n, i}), S_n(J_{n, j}) \big) \\
	&= 
	\frac{2}{n p(u_n)} \sum_{i = 1}^{k_n} \sum_{j = 1}^{i - 1}  \sum_{s \in J_{n, i}} \sum_{t \in J_{n, j}} \Cov \big( \one \big\{ u_n^{-1}\Xvec{s} \in C \big\}, \one \big\{ u_n^{-1}\Xvec{t} \in C \big\}\big)
\end{align*}
As the index set $T$ is bounded and the block distance $\blarge - \bsmall$ between two subsequent small blocks tends to infinity as $n \to \infty$, the index set $s+T$ and $t + T$ in the sum above are all disjoint for sufficiently large $n$.
Hence, we can apply the first part of  \Cref{lem:aux-limits}. 
Counting the pairwise distances between elements of the blocks $J_{n, i}$ and $J_{n, j}$ we get that there is a constant $K > 0$ such that
\begin{align*}
	P_2
	\leq 
	2 K \frac{u_n^{-\alpha}}{p(u_n)} \frac{\bsmall}{\blarge} \sum_{h = \blarge - \bsmall}^{n} (2 - \theta_h)
	\rightarrow 0
\end{align*}
as $n \rightarrow\infty$ since $\sum_{h = 1}^{\infty} (2 - \theta_h) < \infty$. 
Consequently, we have proven \eqref{eq: small-blocks-negligible} and the limiting distribution of $S_n$ and $\sum_{i = 1}^{k_n} S_n(\widetilde{I}_{n, i})$ are equal if the limit exists.

Now, let us consider iid.~copies $\widetilde{S_n}(\widetilde{I}_{n, l})$, $l = 1, \ldots, k_n$, of $S_n(\widetilde{I}_{n, 1})$.
Using a telescoping series we get for any $t \in \mathbb{R}$ that
\begin{align}
	&\bigg\vert
	\mathbb{E} \bigg[
	\prod_{l = 1}^{k_n} e^{itS_n(\widetilde{I}_{n, l})}
	\bigg] 
	-
	\mathbb{E} \bigg[
	\prod_{l = 1}^{k_n} e^{it\widetilde{S}_n(\widetilde{I}_{n, l})}
	\bigg]
	\bigg\vert \notag \\
	&= 
	\bigg\vert
	\sum_{l = 1}^{k_n}
	\mathbb{E}
	\bigg(  \prod_{s = 1}^{l - 1} e^{itS_n(\widetilde{I}_{n, s})} \bigg)
	\bigg( e^{itS_n(\widetilde{I}_{n, l})} - e^{it\widetilde{S}_n(\widetilde{I}_{n, l})}  \bigg)
	\bigg(  \prod_{s = l + 1}^{k_n} e^{it\widetilde{S}_n(\widetilde{I}_{n, s})} \bigg)
	\bigg\vert \notag \\
	&\leq 
	\sum_{l = 1}^{k_n}
	\mathbb{E}
	\bigg\vert
	\bigg(  \prod_{s = 1}^{l - 1} e^{itS_n(\widetilde{I}_{n, s})} \bigg)
	\bigg( e^{itS_n(\widetilde{I}_{n, l})} - e^{it\widetilde{S}_n(\widetilde{I}_{n, l})}  \bigg)
	\bigg\vert \notag \\
	&=
	\sum_{l = 1}^{k_n} \bigg\vert \Cov\bigg( 
	\exp \bigg\{ it \sum_{s = 1}^{l - 1} S_n(\widetilde{I}_{n, s})\bigg\},
	\exp \Big\{ it S_n(\widetilde{I}_{n, l}) \Big\}
	\bigg) \bigg\vert \notag\\
	&\leq
	\sum_{l = 1}^{k_n} \bigg[
	\Big\vert \Cov\bigg( 
	\cos \Big( t \sum_{s = 1}^{l - 1} S_n(\widetilde{I}_{n, s})\Big),
	\cos \Big( t S_n(\widetilde{I}_{n, l}) \Big)
	\bigg) \Big\vert \notag\\
	&\qquad+
	\Big\vert \Cov\bigg( 
	\cos \Big( t \sum_{s = 1}^{l - 1} S_n(\widetilde{I}_{n, s})\Big),
	\sin \Big( t S_n(\widetilde{I}_{n, l}) \Big)
	\bigg) \Big\vert \notag\\
	&\qquad+
	\Big\vert \Cov\bigg( 
	\sin \Big( t \sum_{s = 1}^{l - 1} S_n(\widetilde{I}_{n, s})\Big),
	\cos \Big( t S_n(\widetilde{I}_{n, l}) \Big)
	\bigg) \Big\vert \notag\\
	&\qquad+
	\Big\vert \Cov\bigg( 
	\sin \Big( t \sum_{s = 1}^{l - 1} S_n(\widetilde{I}_{n, s})\Big),
	\sin \Big( t S_n(\widetilde{I}_{n, l}) \Big)
	\bigg) \Big\vert 
	\bigg]
	\label{eq: diff-of-char-fcts}
\end{align}
Next, let us bound each of those covariances.
To do so, let us rewrite
\begin{align}
	&\Cov\bigg( 
	\cos \Big( t \sum_{s = 1}^{l - 1} S_n(\widetilde{I}_{n, s})\Big),
	\cos \Big( t S_n(\widetilde{I}_{n, l}) \Big)
	\bigg) \notag \\
	&= 
	\Cov \bigg( 
	f \Big\{ \frac{t}{\sqrt{n p(u_n)}} \one \{  u_n^{-1} \Xvec{} \in C\} ,\ t \in \bigcup_{s = 1}^{l - 1} \widetilde{I}_{n, s}\Big\}, \notag\\
	&\qquad \qquad \widetilde{f} \Big\{ \frac{t}{\sqrt{n p(u_n)}} \one\{  u_n^{-1} \Xvec{} \in C\} ,\ t \in \widetilde{I}_{n, l}\Big\}
	\bigg) \label{eq: covariance-rewritten}
\end{align}
where $f(x_1, \ldots, x_u) = \cos\big(\sum_{k = 1}^{u} x_k - A_1\big)$, $\widetilde{f}(x_1, \ldots, x_v) = \cos\big(\sum_{k = 1}^{v} x_k - A_2\big)$ for some constants $A_1, A_2 > 0$. 
Clearly, $f$ and $\widetilde{f}$ are bounded Lipschitz functions with Lipschitz constants $\Lip(f), \Lip(\widetilde{f}) \leq 1$.
Furthermore, we know that the sequence $\big\{ X_t, t \in \mathbb{Z} \big\}$ is associated as a max-stable process. 
Now, it is easy to show that $\big\{ \Xvec{}, t \in \mathbb{Z} \big\}$ is associated as well by applying the definition of association.
Furthermore, it follows from \cite[Thm.~1.1.18(d)]{Bulinskii.2007} that $$\big\{  \frac{t}{\sqrt{n p(u_n)}} \one \{  u_n^{-1} \Xvec{} \in C\}, t \in \mathbb{Z} \big\}$$ is associated.
Thus, we can apply Theorem 1.5.3 from \cite{Bulinskii.2007} and bound covariance \eqref{eq: covariance-rewritten} by
\begin{align*}
	&\frac{t^2}{n p(u_n)} 
	\sum_{k_1 \in \bigcup_{s = 1}^{l - 1} \widetilde{I}_{n, s}} \sum_{k_2 \in \widetilde{I}_{n, l}} 
	\bigg\vert
	\Cov \Big( 
	\one \big\{ u_n^{-1} \Xvec{k_1} \in C \big\},
	\one \big\{ u_n^{-1} \Xvec{k_2} \in C \big\}
	\Big)
	\bigg\vert \\
	&\leq
	2Kt^2 \frac{u_n^{-\alpha}}{p(u_n)} \frac{\blarge - \bsmall}{n} 
	\sum_{k = 0}^{l - 1} \sum_{j = 1}^{\blarge - \bsmall}
	(2 - \theta_{k \blarge + \bsmall + j}),
\end{align*}
where $K$ is a positive constant.
Notice that we have used \Cref{lem:aux-bounds} here  and counted the pairwise distances again.
Thus, we can bound the sum of covariances \eqref{eq: diff-of-char-fcts} via
\begin{align*}
	\eqref{eq: diff-of-char-fcts}
	&\leq 8Kt^2 \frac{u_n^{-\alpha}}{p(u_n)} \frac{\blarge - \bsmall}{n} 
	\sum_{l = 1}^{k_n }\sum_{k = 0}^{l - 1} \sum_{j = 1}^{\blarge - \bsmall}
	(2 - \theta_{k \blarge + \bsmall + j}) \\
	&\leq
	8Kt^2 \frac{u_n^{-\alpha}}{p(u_n)} \frac{\blarge - \bsmall}{n} 
	k_n \sum_{j = \bsmall}^{n}
	(2 - \theta_{j}) \\
	&=
	8Kt^2 \frac{u_n^{-\alpha}}{p(u_n)}
	\bigg( 1 - \frac{\bsmall}{\blarge} \bigg)
	\sum_{j = \bsmall}^{n}
	(2 - \theta_{j}) \rightarrow 0
\end{align*}
as $n\rightarrow \infty$.
Therefore, $\sum_{l = 1}^{k_n} S_n(\widetilde{I}_{n, l})$ and $\sum_{l = 1}^{k_n} \widetilde{S_n}(\widetilde{I}_{n, l})$ have the same limiting distributions (provided that the limit exists).

Now, let $\widetilde{S}_n(I_{n, l})$, $l = 1, \ldots, k_n$ be iid copies of $S_n(I_{n, 1})$.
By the same arguments as before it suffices to show
\begin{align*}
	\sum_{l = 1}^{k_n} \widetilde{S_n}(I_{n, l}) \dconv \mathcal{N} (0, \sigma_{T}^2)
\end{align*}
as $n \rightarrow \infty$. 
First, notice that the limiting variance is correct by \Cref{prop:phun} since 
\begin{align*}
	\Var \bigg( \sum_{l = 1}^{k_n} \widetilde{S_n}(I_{n, l}) \bigg) 
	= 
	k_n \Var\big( \widetilde{S_n}(I_{n, 1}) \big)
	=
	\Var \bigg( 
	\sqrt{\blarge p(u_n)} \frac{\multest{}}{p(u_n)}
	\bigg).
\end{align*}
Thus, we need to check the Lindeberg condition of the CLT for triangular arrays of iid.~mean-zero random variables, i.e.~we need to verify 
\begin{align*}
	k_n \mathbb{E} \Big[
	\widetilde{S_n}(I_{n, 1})^{2}\,
	\one \big\{ \vert \widetilde{S_n}(I_{n, 1}) \vert > \varepsilon \big\}
	\Big] \rightarrow 0, \qquad (n \rightarrow \infty),
\end{align*}
for any $\varepsilon > 0$.
Using Chebyshev's inequality and the fact that $|\widetilde{S_n}(I_{n, 1})| \leq \blarge / \sqrt{n p(u_n)} $, we get
\begin{align*}
	k_n \mathbb{E} \Big[
	\widetilde{S_n}(I_{n, 1})^{2}\,
	\one \big\{ \vert \widetilde{S_n}(I_{n, 1}) \vert > \varepsilon \big\}
	\Big]
	\leq 
	\frac{\blarge^2}{n p(u_n)} \frac{k_n \Var\big(\widetilde{S_n}(I_{n, 1})\big)}{\varepsilon^2}.
\end{align*}
Since we have assumed that $\blarge^2 / (np(u_n)) \rightarrow 0$ as $n \rightarrow \infty$, the RHS converges to zero as $n \rightarrow \infty$.
This completes the proof. \qed

\end{document}